\documentclass[final,leqno]{siamltex704}
\usepackage{amsmath,amsfonts,amssymb}
\usepackage{graphicx}
\usepackage[notcite,notref]{showkeys}
\usepackage{mathrsfs}
\usepackage{appendix}
\usepackage{color}
\usepackage{lineno}
\usepackage{float}
\usepackage{dsfont}
\usepackage{pifont}
\usepackage{wrapfig} 
\usepackage{hyperref}
\usepackage{multirow}
\usepackage{bm} 
\numberwithin{equation}{section}
\def\3bar{{|\hspace{-.02in}|\hspace{-.02in}|}}
\def\E{{\mathcal{E}}}
\def\T{{\mathcal{T}}}

\def\pT{{\partial T}}

\def\W{{\mathcal{W}}}

\def\bn{{\bm{n}}}

\newtheorem{algorithm}{Algorithm}[section]

\title {Low Regularity Primal-Dual Weak Galerkin Finite Element Methods for Ill-Posed Elliptic Cauchy Problems}

\author{Chunmei Wang \thanks{Department of Mathematics \& Statistics, Texas Tech University, Lubbock, TX 79409, USA (chunmei.wang@ttu.edu). The research of Chunmei Wang was partially supported by National Science Foundation Award DMS-1849483.}  
}

\begin{document}

\maketitle

\begin{abstract}
A new primal-dual weak Galerkin (PDWG) finite element method is introduced and analyzed for the ill-posed elliptic Cauchy problems with ultra-low regularity assumptions on the exact solution. The Euler-Lagrange formulation resulting from the PDWG scheme yields a system of equations involving both the primal equation and the adjoint (dual) equation. The optimal order error estimate for the primal variable in a low regularity assumption is established. A series of numerical experiments are illustrated to validate effectiveness of the developed theory.
\end{abstract}



\pagestyle{myheadings}
\section{Introduction}
In this paper we consider the ill-posed elliptic Cauchy model problem: Find an unknown function $u$ satisfying
\begin{equation}\label{model}
\begin{split}
-\nabla\cdot (a \nabla u)=&f, \ \qquad
\text{in}\quad \Omega,\\ u=&g_1, \qquad \text{on}\quad \Gamma_D,\\ 
a\nabla u  \cdot \bn=&g_2, \qquad \text{on}\quad \Gamma_N,
\end{split}
\end{equation}
where $\Omega\subset \mathbb R^d (d=2, 3)$ is an open bounded domain with Lipschitz continuous boundary $\partial \Omega$, $\Gamma_D$ and $\Gamma_N$ are two segments of the domain boundary, $f\in L^2(\Omega)$, the Cauchy data $g_1\in H^{\frac{1}{2}}(\Gamma_D)$ and $g_2\in (H_{00}^{\frac{1}{2}}(\Gamma_N))'$,  the coefficient tensor $a(x)$ is symmetric, bounded, and uniformly positive definite in the domain $\Omega$, and $\bn$ is the unit outward normal vector to $\Gamma_N$.
The elliptic Cauchy problem is to solve partial differential equations (PDEs) in a domain where over-specified boundary conditions are given on parts of the domain boundary. The elliptic Cauchy problem is also to solve a data completion problem with missing boundary data on the remaining parts of the domain boundary.

The elliptic Cauchy problem  arises in science and engineering, e.g., vibration, wave propagation, cardiology, electromagnetic scattering, geophysics, nondestructive testing and steady-state inverse heat conduction. In particular, the Cauchy problem for second order elliptic equations plays an important role in the inverse boundary value problems modeled by elliptic PDEs.  Readers are referred to \cite{s1, s2, s3, s4, s5, s6, s7, s8, s9, s10, s11, s12, k6, k10, k11} and the references cited therein for details of the elliptic Cauchy problems.

There has been a long history tracing back to Hadamard \cite{e17, k22, h1, h2, h3} for the study of the elliptic Cauchy problem \eqref{model}. When it comes to the case of $\Gamma_D=\Gamma_N$, Hadamard demonstrated the ill-posedness of the problem \eqref{model} by constructing an example where the solution does not depend continuously on the Cauchy data. Hadamard and others \cite{s13, s14, s15} found that a small perturbation in the data might result in an enormous error in the numerical solution for elliptic Cauchy problem. The Schwartz reflection principle \cite{Gilbarg-Trudinger} indicates that in most cases the existence of solutions for the model problem \eqref{model} can not be guaranteed  for any given Cauchy data $g_1$ and $g_2$. However, \cite{Andrieux} showed that the elliptic Cauchy problem \eqref{model} has a solution for any given Cauchy data $g_1\times g_2\in M$ where $M$ is a dense subset of $H^{\frac12}(\Gamma_D)\times [H_{00}^{\frac12}(\Gamma_N)]^\prime$. It is well-known that the solution of the elliptic Cauchy problem \eqref{model} (if it exists) must be unique, provided that $\Gamma_D\cap\Gamma_N$ is a nontrivial portion of the domain boundary. The Cauchy data is thus assumed to be compatible such that the solution exists. Throughout this paper, we assume that $\Gamma_D\cap\Gamma_N$ is a nontrivial portion of the domain boundary so that the solution (if it exists) of the elliptic Cauchy problem \eqref{model} is unique.

In the literature, there are two main numerical strategies developed for the elliptic Cauchy problem: (1) Tikhonov regularization is applied to the problem with missing boundary data to determine the solution; (2) A sequence of well-posed problems in the same equation is  iteratively employed to approximate the ill-posed problem. \cite{s16} developed the numerical method for the elliptic Cauchy problem  based on the tools of boundary integral equations, single-layer potential function and jump relations. \cite{Falk-Monk} introduced an optimization approach based on least squares and Tikhonov regularization techniques. The  finite element method based on an optimal control characterization of the Cauchy problem was analyzed in \cite{r10}. The stabilized finite element method \cite{Burman, ErikBurman-EllipticCauchy} based on a general framework involving both the original equation and its adjoint equation is applicable to a wide class of ill-posed problems where only weak continuity is necessary. More numerical methods were proposed and analyzed for the elliptic Cauchy problem including the conjugate gradient boundary element method, the boundary knot method, the alternating iterative boundary element method, the moment method, the boundary particle method, the method of level set type,  and the method of fundamental solutions \cite{Leitao, s4, s17, s18, s19, s20, s21,s22,s12}. Other theoretical and applied work have also been developed such as regularization methods \cite{r5, r6}, Steklov-Poincare theory \cite{r2, r3, Belgacem}, minimal error methods  \cite{r8, r9} and quasi-reversibility methods \cite{Bourgeois}.

This paper is devoted to the development of a new primal-dual weak Galerkin finite element method for the elliptic Cauchy model problem \eqref{model}. The PDWG framework provides mechanisms to enhance the stability of a numerical scheme by combining solutions of the primal and the dual (adjoint) equation.  PDWG methods have been successfully applied to solve the second order elliptic equation in non-divergence form \cite{ww2016}, the elliptic Cauchy problem \cite{w2018,ww2018}, the Fokker-Planck type equation \cite{ww2017}, the convection diffusion equation \cite{cao, ludmil}, and the transport equation \cite{tr1, tr2}. The PDWG method has the following advantages over other existing schemes: (1) it offers a symmetric and well-posed problem for the ill-posed elliptic Cauchy problem; (2) it is consistent in the sense that the system is satisfied by the exact solution (if it exists);  (3) it is applicable to a wide class of PDE problems for which no traditional variational formulation is available; and (4) it admits general finite element partitions consisting of arbitrary polygons or polyhedra. The main contribution of this paper lies in two aspects: (1) the development of a new PDWG scheme that admits boundary data with low regularity due to noise or uncertainties; and (2) the establishment of a mathematical convergence theory with optimal order error estimates under low regularity assumptions for the exact solution.

Throughout the paper, we use the standard notations for Sobolev spaces
and norms. For any open bounded domain $D\subset \mathbb{R}^d$ with
Lipschitz continuous boundary,  denote by $\|\cdot\|_{s, D}$, $|\cdot|_{s, D}$ and $(\cdot,\cdot)_{s, D}$ the norm, seminorm and the inner product in the Sobolev space $H^s(D)$ for $s\ge 0$, respectively. The norms in $H^{s}(D)$ ($s<0$) are defined by duality with the norms in $H^{|s|}(D)$. The space $H^0(D)$ coincides with $L^2(D)$, where the norm and the inner product are denoted by $\|\cdot \|_{D}$ and $(\cdot,\cdot)_{D}$, respectively. When $D=\Omega$, or when the domain of integration is clear from the context,  the subscript $D$ is dropped in the norm and the inner product notation. For convenience,  we use ``$\lesssim$'' to denote ``less than or equal to'' up to a generic constant which is independent of important parameters such as the mesh size and physical parameters.

The paper is organized as follows. In Section 2, the weak formulation of the elliptic Cauchy model problem is proposed and the discrete weak differential operator is briefly reviewed. In Section 3, the PDWG algorithm for the elliptic Cauchy model problem is introduced. Section 4 demonstrates the solution existence and uniqueness of the proposed PDWG scheme. The error equations for the PDWG scheme are derived in Section 5. In Section 6, the optimal order error estimates are established for the PDWG method in some discrete Sobolev norms. In Section 7, a series of numerical results are reported to demonstrate the effectiveness of the PDWG method developed in the previous sections.

\section{Weak Formulations and Discrete Weak Differential Operators} This section will introduce the weak formulation of the elliptic Cauchy model problem \eqref{model} and briefly review the discrete weak differential operator \cite{wy3655}.

Denote by $\Gamma_N^c=\partial \Omega \setminus \Gamma_N$ and 
$\Gamma_D^c=\partial \Omega \setminus \Gamma_D$.
The weak formulation of the elliptic Cauchy model problem
\eqref{model} seeks $u\in L^2(\Omega)$ satisfying
\begin{equation}\label{weakform}
 (u,  \nabla \cdot (a\nabla \sigma)) =-(f, \sigma)-\langle g_2, \sigma\rangle_{\Gamma_N}+\langle g_1, a\nabla \sigma \cdot \bn\rangle_{\Gamma_D}, \quad \forall \sigma\in W,
\end{equation}
where $W=\{\sigma\in H^1(\Omega), a\nabla \sigma \in H(\operatorname{div}; \Omega),  \sigma|_{\Gamma_N^c}=0, a\nabla \sigma \cdot \bn|_{\Gamma_D^c}=0\}.$


Let ${\cal T}_h$ be a partition of the domain $\Omega$ into polygons in 2D or polyhedra in 3D which is shape regular in the sense of \cite{wy3655}. Denote by ${\mathcal E}_h$ the set of all edges or flat faces in ${\cal T}_h$ and  ${\mathcal E}_h^0={\mathcal E}_h \setminus \partial\Omega$ the set of all interior edges or flat faces. Denote by $h_T$ the meshsize of $T\in {\cal T}_h$ and
$h=\max_{T\in {\cal T}_h}h_T$ the meshsize for the partition ${\cal T}_h$. For the convenience of analysis and without loss of generality, in what follows of this paper,  we assume that the coefficient tensor $a$ is piecewise constants with respect to the partition $\T_h$. The analysis could be easily generalized to the case that the coefficient tensor $a$ is piecewise smooth functions with respect to the partition $\T_h$. 

Let $T\in \mathcal{T}_h$ be a polygonal or polyhedral region with
boundary $\partial T$. For any $\phi\in H^{1-\theta}(T)$ ($0\leq \theta <1$) and any polynomial $\psi$, the following trace inequalities hold true  \cite{wy3655}
\begin{equation}\label{tracein}
 \|\phi\|^2_{\partial T} \lesssim h_T^{-1}\|\phi\|_T^2+h_T^{1-2\theta} \|\phi\|_{1-\theta,T}^2, \qquad
 \|\psi\|^2_{\partial T} \lesssim h_T^{-1}\|\psi\|_T^2.
\end{equation}

A weak function on $T\in\T_h$ is denoted by a triplet $\sigma=\{\sigma_0,\sigma_b, \sigma_n\}$ such that $\sigma_0\in L^2(T)$, $\sigma_b\in L^{2}(\partial T)$ and $\sigma_n\in L^{2}(\partial T)$. The first and the second components, namely $\sigma_0$ and $\sigma_b$, represent the values of $\sigma$ in the interior and on the boundary of $T$ respectively. The third component $\sigma_n$ can be understood as the value of $a\nabla \sigma \cdot \bn$ on $\pT$, where $\bn$ is the unit outward normal vector on $\pT$. 
Note that $\sigma_b$ and $\sigma_n$ may not necessarily be the traces of $\sigma_0$ and $a\nabla \sigma_0 \cdot \bn$ on $\partial T$. Denote by $\W(T)$ the space of all weak functions on $T$; i.e.,
\begin{equation}\label{2.1}
\W(T)=\{\sigma=\{\sigma_0,\sigma_b, \sigma_n \}: \sigma_0\in L^2(T), \sigma_b\in
L^{2}(\partial T), \sigma_n\in L^{2}(\partial T)\}.
\end{equation}

For simplicity, denote by $\mathcal{L}=\nabla \cdot (a\nabla)$. The weak $\mathcal{L}$ operator of $\sigma\in \W(T)$, denoted by $\mathcal{L}_w  \sigma$, is defined as a linear functional such that
 \begin{eqnarray*}
(\mathcal{L}_w \sigma, \phi)_T & :=& (\sigma_0, \mathcal{L} \phi)_T -\langle \sigma_b, a \nabla \phi\cdot \bm{n}\rangle_{\partial T} +
\langle \sigma_n,\phi\rangle_{\partial T},
\end{eqnarray*}
for all $\phi$ satisfying $\phi\in H^1(\Omega)$ and $a \nabla \phi \in H(div; T)$. 

Denote by $P_r(T)$ the space of polynomials on the element $T$ with degree no more than $r$. A discrete version of $\mathcal{L}_w \sigma$, denoted by $\mathcal{L}_{w, r, T}\sigma$, is defined as the unique polynomial in $P_r(T)$ satisfying
\begin{eqnarray}\label{disvergence}
 (\mathcal{L}_{w,r,T} \sigma, w)_T &=& 
 (\sigma_0, \mathcal{L} w)_T-\langle \sigma_b, a \nabla w\cdot
\bm{n}\rangle_{\partial T}+ \langle \sigma_n,w\rangle_{\partial T},  \forall w\in P_r(T),
\end{eqnarray} 
which, using the usual integration by parts, gives \begin{eqnarray}\label{disgradient*}
   (\mathcal{L}_{w,r,T} \sigma, w)_T
  & = & (\mathcal{L}\sigma_0, w)_T+\langle \sigma_0-\sigma_b, a \nabla w\cdot
\bm{n}\rangle_{\partial T}- \langle a\nabla \sigma_0\cdot \bn-\sigma_n, w \rangle_{\partial T}.
\end{eqnarray} 

\section{Primal-Dual Weak Galerkin Scheme}\label{Section:WGFEM}
For any given integer $k\geq 2$, let $W_k(T)$ be the local discrete weak function space; i.e.,
$$
W_k(T)=\{\{\sigma_0, \sigma_b, \sigma_n\}: \sigma_0\in P_k(T), \sigma_b\in P_k(e),  \sigma_n \in P_{k-1}(e),e\subset \partial T\}.
$$
Patching $W_k(T)$ over all the elements $T\in {\cal T}_h$
through a common value $\sigma_b$ on the interior interface $\E_h^0$, we obtain a global weak finite element space $W_h$; i.e.,
$$
W_h=\big\{\{\sigma_0,\sigma_b, \sigma_n\}:\{\sigma_0,\sigma_b, \sigma_n\}|_T\in W_k(T), \forall T\in \mathcal{T}_h \big\}.
$$
We further introduce the subspace of $W_h$ with homogeneous Dirichlet  and Neumann boundary conditions, denoted by $W_h^0$; i.e.,
$$
W_h^0=\{v\in W_h: \sigma_b=0 \ \text{on}\ \Gamma_N^c, \sigma_n=0 \ \text{on}\ \Gamma_D^c\}.
$$
In addition, let $M_h$ be the finite element space consisting of piecewise polynomials of degree $k-2$; i.e.,
$$
M_h=\{w: w|_T\in P_{k-2}(T),  \forall T\in \mathcal{T}_h\}.
$$

For simplicity,  for any $\sigma=\{\sigma_0, \sigma_b, \sigma_n\}\in
W_h$, denote by $\mathcal{L}_w \sigma$ the discrete weak operator
$\mathcal{L}_{w, k-2, T} \sigma$ computed  by using
(\ref{disvergence}) on each element $T$; i.e.,
$$
(\mathcal{L}_w \sigma)|_T=\mathcal{L}_{w ,k-2,T}(\sigma|_T), \qquad \forall \sigma\in W_h.
$$
 
For any $ \lambda, \sigma \in W_h$, and $u\in M_h$, we introduce the following bilinear forms
\begin{eqnarray*}
s(\lambda, \sigma)&=&\sum_{T\in \mathcal{T}_h} s_T(\lambda, \sigma),\label{EQ:s-form}\\
b(u, \sigma)&=&\sum_{T\in \mathcal{T}_h}(u, \mathcal{L}_w \sigma)_T,\label{EQ:b-form}
\end{eqnarray*}
where
\begin{equation*}\label{EQ:sT-form}
\begin{split}
s_T(\lambda, \sigma)=&h_T^{-1}\langle \lambda_0-\lambda_b, \sigma_0-\sigma_b\rangle_{\partial T}
 +h_T\langle  a \nabla \lambda_0 \cdot \bn-\lambda_n,  a \nabla \sigma_0 \cdot \bn-\sigma_n\rangle_{\partial T}.\\
\end{split}
\end{equation*}

The primal-dual weak Galerkin finite element scheme based on the weak formulation \eqref{weakform} for the elliptic Cauchy model problem \eqref{model} is described as follows:
\begin{algorithm}[PDWG Scheme]
  Find $(u_h;\lambda_h)\in M_h
\times W_{h}^0$ satisfying
\begin{eqnarray}\label{32}
s(\lambda_h, \sigma)+b(u_h,\sigma)&=&-(f,\sigma_0)- \langle g_2, \sigma_b \rangle_{\Gamma_N}+\langle g_1, \sigma_n\rangle_{\Gamma_D},\quad  \forall \sigma\in W^0_h,\\
b(v, \lambda_h)&=& 0,  \qquad\qquad\qquad\qquad\qquad\qquad\qquad\qquad \forall v\in M_{h}.\label{2}
\end{eqnarray}
\end{algorithm}

On each element $T$, denote by $Q_0$ the $L^2$ projection operator onto $P_k(T)$; on each edge or face $e\subset\partial T$, denote by $Q_b$ and $Q_n$ the $L^2$ projection operators onto $P_{k}(e)$ and $P_{k-1}(e)$,
respectively. For any $w\in H^1(\Omega)$, denote by $Q_h w$ the $L^2$ projection
onto the weak finite element space $W_h$ such that on each element
$T$,
$$
Q_hw=\{Q_0w,Q_bw, Q_n(a\nabla w \cdot \bn)\}.
$$
Denote by $\mathcal{Q}^{k-2}_h$ the $L^2$ projection operator onto the space $M_h$.
 
\begin{lemma}\label{Lemma5.1} \cite{ludmil}  The $L^2$ projection operators $Q_h$ and $\mathcal{Q}^{k-2}_h$ satisfy the following commuting property:
 \begin{equation}\label{div}
\mathcal{L}_{w}(Q_h w) = \mathcal{Q}_h^{k-2}( \mathcal{L} w),   \qquad \forall w \in H^1(T), \; a\nabla w \in H(\operatorname{div}; T).
\end{equation}
\end{lemma}

\section{Existence and Uniqueness}
In this section, we shall establish the solution existence and uniqueness of the PDWG scheme (\ref{32})-(\ref{2}).

In the finite element space $W_h$, we introduce a semi-norm induced from the stabilizer; i.e., 
\begin{equation}\label{EQ:triplebarnorm}
\3bar \sigma \3bar = s(\sigma, \sigma)^{\frac{1}{2}}, \qquad \forall  \sigma\in W_h.
\end{equation}
We introduce a semi-norm for the finite element space $M_h$; i.e.,
\begin{equation}\label{EQ:triplebarnorm2}
\3bar v\3bar_1=\Big(\sum_{T\in \mathcal{T}_h}h_T^2\|\mathcal{L} v\|^2_T+\sum_{e\in \mathcal{E}_h}h_T\|[[a \nabla v\cdot \bn]]\|^2_{e}+\sum_{e\in \mathcal{E}_h}h_T^{-1}\|[[v]]\|^2_{e}\Big)^{\frac{1}{2}}, \  \forall v\in M_h, 
\end{equation}
where $[[v]]=v|_{\pT_1}-v|_{\pT_2}$ is the jump of $v$ on the interior edge $e=T_1 \cap T_2 \in {\cal E}_h^0$ and $[[v]]=v$ on the boundary edge $e\subset \partial \Omega$; the same definition applies to $[[a \nabla v\cdot \bn]]$ with $\bn$ being the unit outward normal direction to $\pT$.

\begin{lemma}\label{Lemma:8.1}\cite{wy3655}
Let $k\geq 2$ and $\mathcal{T}_h$ be a shape regular partition of the domain $\Omega$ specified in \cite{wy3655}. For $0\leq t \leq 2$, the following estimates hold true:
\begin{eqnarray} 
& \sum_{T\in \mathcal{T}_h}h_T^{2t}\|u-{\cal
Q}^{(k-2)}_hu\|^2_{t,T}  \lesssim h^{2m}\|u\|^2_{m},&\qquad m \in [t, k-1].\label{term3}
\end{eqnarray}
\end{lemma}

\begin{lemma}\label{unilem} \cite{w2018, ww2018} Let $\Omega$ be an open bounded and connected domain in $\mathbb R^d \ (d=2,3)$ with Lipschitz continuous domain boundary $\partial\Omega$.  Assume that $\Gamma_D\cap \Gamma_N$ is a non-trivial portion of the domain boundary $\partial\Omega$. The solution  of the elliptic Cauchy problem \eqref{model} (if it exists) is unique.
\end{lemma}

\begin{lemma} \label{lem3-new} 
The following {\it inf-sup} condition holds true. For any $v\in M_h$, there exists $\rho_v\in W_h^0$ satisfying
\begin{eqnarray} \label{EQ:inf-sup-condition-01}
 b(v, \rho_v) = \3bar v\3bar_1^2, \quad \3bar \rho_v \3bar\leq \beta \3bar v\3bar_1,
\end{eqnarray}
where $\beta>0$ is a constant independent of the meshsize $h$.
\end{lemma}
\begin{proof} For  any $v\in M_h$,  we choose $\rho \in W_h^0$ such that on each element $T$, $\rho=\{h_T^2{\cal L} v, -h_e [[a \nabla v\cdot \bn]], h_e^{-1}[[v]]\}$, with boundary value properly adjusted to satisfy the homogeneous boundary data. From the trace inequality (\ref{tracein}), triangle inequality, and \eqref{EQ:triplebarnorm2}, we have
\begin{equation}\label{EQ:Estimate:002}
\begin{split}
&\sum_{T\in \mathcal{T}_h }h_T^{-1}\|\rho_0-\rho_b\|_{\pT}^2\\
=&\sum_{T\in \mathcal{T}_h}
h_T^{-1}\|h_T^2{\cal L} v+h_e [[a \nabla v\cdot \bn]]\|_{\pT}^2\\
\lesssim &\sum_{T\in \mathcal{T}_h}
\Big(h_T^3\|{\cal L} v\|_{\pT}^2+h_T\| [[a \nabla v\cdot \bn]]\|_{\pT}^2\Big)\\
\lesssim & \sum_{T\in \mathcal{T}_h}\Big( h_T^2\|{\cal L} v\|_{T}^2+h_T\|[[a \nabla v\cdot \bn]]\|_{\pT}^2\Big)\\
\lesssim & \ \3bar v\3bar_1^2.
 \end{split}
\end{equation} 
Analogously, using the trace inequality \eqref{tracein}, inverse inequality, triangle inequality, and \eqref{EQ:triplebarnorm2} gives
 \begin{equation}\label{EQ:Estimate:003}
\begin{split}
&\sum_{T\in \mathcal{T}_h }h_T\|a\nabla
\rho_0\cdot\bn-[[v]]\|_{\pT}^2\\
= & \sum_{T\in \mathcal{T}_h }h_T\|a\nabla
(h_T^2{\cal L} v)\cdot\bn-h_T^{-1}[[v]]\|_{\pT}^2\\
\lesssim & \sum_{T\in \mathcal{T}_h }h_T\Big(h_T^4 \|a\nabla
({\cal L} v)\cdot\bn\|_{\pT}^2+h_T^{-2}\|[[v]]\|_{\pT}^2\Big)\\
\lesssim & \sum_{T\in \mathcal{T}_h }\Big(h_T^{2}\| 
{\cal L} v \|_{T}^2+h_T^{-1}\|[[v]]\|_{\pT}^2\Big)\\
\lesssim & \ \3bar v\3bar_1^2.
\end{split}
\end{equation}

From \eqref{EQ:triplebarnorm}-\eqref{EQ:triplebarnorm2}, combining the estimates (\ref{EQ:Estimate:002})-(\ref{EQ:Estimate:003}) yields
\begin{equation}\label{EQ:inf-sup-condition-02}
\3bar \rho\3bar \lesssim \3bar v\3bar_1.
\end{equation}

Next, using (\ref{disvergence}) yields \begin{equation}\label{EQ:April05:100}
\begin{split}
b(v, \rho) = & \sum_{T\in \mathcal{T}_h}(v, \mathcal{L}_w \rho)_T  \\
 =&\sum_{T\in \mathcal{T}_h}(\rho_0,  \mathcal{L} v)_T-\langle \rho_b, a \nabla v\cdot \bn\rangle_{\pT}+\langle \rho_n, v\rangle_{\pT}\\
  =&\sum_{T\in \mathcal{T}_h}(\rho_0,  \mathcal{L} v)_T-\sum_{e\in \mathcal{E}_h}\langle \rho_b,  [[a \nabla v\cdot \bn]]\rangle_{e}+\sum_{e\in \mathcal{E}_h}\langle \rho_n, [[v]]\rangle_{e}\\
=&\sum_{T\in \mathcal{T}_h}h_T^2 \|\mathcal{L} v\|^2_T+\sum_{e\in \mathcal{E}_h}h_T\|[[a \nabla v\cdot \bn]]\|^2_{e}+\sum_{e\in \mathcal{E}_h}h_T^{-1}\|[[v]]\|^2_{e}\\
=&\ \3bar v\3bar_1^2,
\end{split}
\end{equation}
which, together with \eqref{EQ:inf-sup-condition-02}, completes the proof of the lemma.
\end{proof}

The following theorem is the main result on solution existence and uniqueness for the PDWG scheme \eqref{32}-\eqref{2}.

\begin{theorem}\label{thmunique1}
Assume that $\Gamma_D\cap\Gamma_N$ contains a nontrivial portion of the domain boundary $\partial\Omega$ and $\Gamma_D\cup \Gamma_N\Subset \partial\Omega$ is a proper closed subset. The PDWG finite element algorithm (\ref{32})-(\ref{2}) has one and only one solution.
\end{theorem}
\begin{proof} It suffices to show that zero is the unique solution to the problem (\ref{32})-(\ref{2}) with homogeneous data $f=0$, $g_1=0$ and $g_2=0$. To this end, assume $f=0$, $g_1=0$ and $g_2=0$ in (\ref{32})-(\ref{2}). By letting $v=u_h$ and $\sigma=\lambda_h$, the difference of (\ref{2}) and (\ref{32}) gives $s(\lambda_h,\lambda_h)=0$, which implies $\lambda_0=\lambda_b$ and $a\nabla \lambda_0 \cdot \bn=\lambda_n$ on each $\partial T$. Therefore, together with the fact that $\lambda_h\in W_h^0$, we obtain $\lambda_0=0$ on $\Gamma_N^c$ and $a\nabla \lambda_0 \cdot \bn=0$ on $\Gamma_D^c$.

It follows from (\ref{2}), (\ref{disgradient*})  and the usual integration by parts that for any $v\in M_h$
\begin{equation}
\begin{split}
0=&b(v, \lambda_h)\\
=&\sum_{T\in \mathcal{T}_h}(v,  \mathcal{L}_w \lambda_h)_T \\
=&\sum_{T\in \mathcal{T}_h}  (\mathcal{L}\lambda_0, v)_T+\langle \lambda_0- \lambda_b, a \nabla v\cdot
\bm{n}\rangle_{\partial T}- \langle a\nabla \lambda_0\cdot \bn-\lambda_n, v \rangle_{\partial T}
\\
 =&\sum_{T\in \mathcal{T}_h}  ( \mathcal{L}\lambda_0, v)_T,\\
\end{split}
\end{equation}
where we have used $\lambda_0=\lambda_b$ and $a\nabla \lambda_0 \cdot \bn=\lambda_n$ on each $\partial T$. This implies $\mathcal{L}\lambda_0=0$ on each element $T\in\mathcal{T}_h$ by taking $v= \mathcal{L}\lambda_0$, which further yields $\mathcal{L}\lambda_0=0$ in $\Omega$. Since $\Gamma_D\cup \Gamma_N\Subset \partial \Omega$ is a proper closed subset,  $\Gamma_D^c\cap \Gamma_N^c=(\Gamma_D\cup \Gamma_N)^c$ contains a nontrivial portion of $\partial \Omega$. Since $\lambda_0=0$ on $\Gamma_N^c$ and $a\nabla \lambda_0 \cdot \bn=0$ on $\Gamma_D^c$, from Lemma \ref{unilem}, we get $\lambda_0\equiv 0$ in $\Omega$. It follows that $\lambda_h\equiv0$ in $\Omega$, as $\lambda_b=\lambda_0$ and $\lambda_n=a\nabla \lambda_0 \cdot \bn$ on each $\partial T$.

To demonstrate $u_h\equiv 0$, we use $\lambda_h\equiv 0$ in $\Omega$ and the equation (\ref{32}) to obtain
\begin{equation}\label{rew}
b(u_h, \sigma) = 0,\qquad \forall \sigma\in W_{h}^0,
\end{equation}
which, together with Lemma \ref{lem3-new}, leads to $u_h=0$. 
This yields ${\cal L} u_h=0$ on each element $T\in {\cal T}_h$, $[[a \nabla u_h\cdot \bn]]=0$ and $[[u_h]]=0$ on each $e \in {\cal E}_h^0$, and $a \nabla u_h\cdot \bn=0$ and $u_h=0$ on each $e\subset \partial \Omega$. Therefore, we have  ${\cal L} u_h=0$ in $\Omega$, $a \nabla u_h\cdot \bn=0$ and $u_h=0$ on $\partial\Omega$, for which there exists a unique solution $u_h=0$ in $\Omega$. 

This completes the proof of the theorem.
\end{proof}

\section{Error Equations}\label{Section:error-equations} This section is devoted to deriving the error equations for the PDWG scheme (\ref{32})-(\ref{2}) which will play a critical role in establishing the error estimates in the following section.

Let $u$ and $(u_h, \lambda_h) \in  M_{h} \times W_{h}^0$ be the exact solution of the elliptic Cauchy model problem (\ref{model}) and PDWG solution of the numerical scheme (\ref{32})-(\ref{2}), respectively. Note that the Lagrange multiplier $\lambda_h$ approximates the trivial solution $\lambda=0$. 
The error functions are defined as the difference between the numerical solution $(u_h, \lambda_h)$ and the $L^2$ projection of the exact solution of (\ref{model}); i.e.,
\begin{align}\label{error}
e_h&=u_h-\mathcal{Q}^{k-2}_hu,\\
 \varepsilon_h&=\lambda_h-Q_h\lambda=\lambda_h.\label{error-2}
\end{align}

\begin{lemma}\label{Lemma:LocalEQ} 
For any $\sigma\in W_{h}$ and $v\in M_{h}$, the following identity holds true:
\begin{equation*}
(\mathcal{L}_w \sigma, v)_T = (\mathcal{L} \sigma_0, v)_T +R_T(\sigma, v),
\end{equation*}
where
\begin{equation*}\label{EQ:April:06:100}
\begin{split}
R_T(\sigma,v) = \langle \sigma_0-\sigma_b, a \nabla  v\cdot
\bm{n}\rangle_{\partial T}-\langle a\nabla \sigma_0\cdot \bn-\sigma_n,    v \rangle_{\partial T}.
\end{split}
\end{equation*}
\end{lemma}

\begin{proof}
This proof can be easily obtained using (\ref{disgradient*}).
\end{proof}

\begin{lemma}\label{errorequa}
Let $u$ and $(u_h;\lambda_h) \in M_{h}\times W_{h}^0$ be the solutions arising from (\ref{model}) and (\ref{32})-(\ref{2}), respectively. 
The error functions $e_h$ and $\varepsilon_h$ satisfy the following error equations; i.e.,
\begin{eqnarray}\label{sehv}
s(\varepsilon_h, \sigma)+b(e_h, \sigma)&=&\ell_u(\sigma),\qquad \forall\  \sigma\in W_{h}^0,\\
b(v, \varepsilon_h)&=&0,\qquad\qquad \forall v\in M_{h}, \label{sehv2}
\end{eqnarray}
where $\ell_u(\sigma)$ is given by
\begin{equation}\label{lu}
\begin{split}
\qquad \ell_u(\sigma) =&\sum_{T\in \mathcal{T}_h} \langle \sigma_0-\sigma_b, a \nabla (u-\mathcal{Q}_h^{k-2}u)\cdot \bm{n}\rangle_{\partial T}-\langle a\nabla \sigma_0\cdot \bn-\sigma_n, u-\mathcal{Q}_h^{k-2}u\rangle_{\partial T}.
\end{split}
\end{equation}
\end{lemma}

\begin{proof} From (\ref{error-2}) and (\ref{2}) we have
\begin{align*}
b(v, \varepsilon_h) = b(v, \lambda_h) = 0,\qquad \forall v\in M_{h},
\end{align*}
which gives rise to (\ref{sehv2}).

Recall that $\lambda=0$. From (\ref{32}) we arrive at
\begin{equation}\label{EQ:April:04:100}
\begin{split}
 & s(\lambda_h-Q_h\lambda, \sigma)+b(u_h-\mathcal{Q}^{k-2}_hu, \sigma) \\
   = &-(f, \sigma_0)- \langle g_2, \sigma_b \rangle_{\Gamma_N}+\langle g_1, \sigma_n\rangle_{\Gamma_D}-b(\mathcal{Q}^{k-2}_hu, \sigma).
\end{split}
\end{equation}
As to the term $b(\mathcal{Q}^{k-2}_hu, \sigma)$, using Lemma  \ref{Lemma:LocalEQ} and the usual integration by parts gives
\begin{equation}\label{EQ:April:04:103}
\begin{split}
&b(\mathcal{Q}^{k-2}_hu, \sigma) \\
= & \sum_{T\in \mathcal{T}_h} (\mathcal{Q}^{k-2}_hu,  \mathcal{L}_w \sigma)_T \\
=& \sum_{T\in \mathcal{T}_h} (\mathcal{L}\sigma_0, \mathcal{Q}_h^{k-2}u)_T +R_T(\sigma, \mathcal{Q}_h^{k-2}u)\\
= & \sum_{T\in \mathcal{T}_h} (\mathcal{L}\sigma_0, u)_T   +R_T(\sigma, \mathcal{Q}_h^{k-2}u)\\
=&\sum_{T\in \mathcal{T}_h}  (\sigma_0,  \nabla \cdot (a\nabla u))_T-\langle \sigma_0,  a\nabla u \cdot\bn \rangle_{\partial T}+\langle a\nabla \sigma_0\cdot \bn, u\rangle_{\partial T} +R_T(\sigma, \mathcal{Q}_h^{k-2}u).\\
\end{split}
\end{equation}

Since $u$ is the exact solution of (\ref{model}), $\sigma_b=0$ on $\Gamma_N^c$ and $\sigma_n=0$ on $\Gamma_D^c$, we have
\begin{eqnarray}\label{EQ:April:04:105}
 \sum_{T\in \mathcal{T}_h}\langle \sigma_b,  a\nabla u \cdot\bn \rangle_{\partial T} & = & \langle \sigma_b, g_2 \rangle_{\Gamma_N},\\
\sum_{T\in \mathcal{T}_h}  \langle \sigma_n, u\rangle_{\partial T}& = &
\langle \sigma_n, g_1\rangle_{\Gamma_D}.\label{EQ:April:04:106}
\end{eqnarray}
Substituting  (\ref{EQ:April:04:105}), (\ref{EQ:April:04:106}) and (\ref{model}) into (\ref{EQ:April:04:103}), we arrive at
\begin{equation}\label{EQ:April:04:107}
\begin{split}
&b(\mathcal{Q}^{k-2}_hu, \sigma) \\
= & -(\sigma_0, f)-\sum_{T\in \mathcal{T}_h} \langle \sigma_0-\sigma_b,  a\nabla u \cdot\bn \rangle_{\partial T}+\langle a\nabla \sigma_0\cdot \bn-\sigma_n, u\rangle_{\partial T}\\
& -\langle \sigma_b, g_2 \rangle_{\Gamma_N}+\langle \sigma_n, g_1\rangle_{\Gamma_D}+R_T(\sigma, \mathcal{Q}_h^{k-2}u).
\end{split}
\end{equation}
Substituting (\ref{EQ:April:04:107})  into (\ref{EQ:April:04:100}) gives rise to the error equation (\ref{sehv}), which completes the proof of the lemma.
\end{proof}

\section{Error Estimates}\label{Section:Stability}
In this section, we shall demonstrate the optimal order of error estimates for the PDWG scheme (\ref{32})-(\ref{2}).

\begin{theorem} \label{theoestimate}
 Let $u$ be the exact solution of the elliptic Cauchy model problem (\ref{model}) and $(u_h, \lambda_h) \in M_{h} \times W_{h}^0$ be the numerical solution arising from PDWG method (\ref{32})-(\ref{2}) with $k\ge 2$. Assume (1) the coefficient tensor $a$ is piecewise constants in $\Omega$ with respect to the finite element partition $\mathcal{T}_h$ which is shape regular as specified in \cite{wy3655}; (2) the exact solution $u$ is sufficiently regular such that $u\in H^{m}(\Omega)$ ($2-\theta\leq m \leq k-1$) with $0\leq \theta<\frac12$. The following error estimates hold true; i.e.,
 \begin{equation}\label{erres}
\3bar \lambda_h \3bar+\3bar e_h\3bar_{1} 
\lesssim h^{m-1}\|u\|_{m}.
\end{equation}
\end{theorem}

\begin{proof}
Letting $\sigma=\varepsilon_h=\{\varepsilon_0,\varepsilon_b,\varepsilon_n\}$ in
the error equation (\ref{sehv}) and using (\ref{sehv2}) and \eqref{lu} we arrive at
\begin{equation}\label{EQ:April7:001}
\begin{split}
&s(\varepsilon_h, \varepsilon_h) = \ell_u(\varepsilon_h)\\
   =&\sum_{T\in \mathcal{T}_h} \langle \varepsilon_0-\varepsilon_b, a \nabla ( u-\mathcal{Q}_h^{k-2}u) \cdot
\bm{n}\rangle_{\partial T}  +\langle \varepsilon_n - a\nabla \varepsilon_0\cdot \bn, u-\mathcal{Q}_h^{k-2}u\rangle_{\partial T}.\\
\end{split}
\end{equation}

Using Cauchy-Schwarz inequality, the trace inequality \eqref{tracein}, the estimate \eqref{term3}  and \eqref{EQ:triplebarnorm} gives
\begin{equation}\label{EQ:J2}
\begin{split} 
 & \left|\sum_{T\in \mathcal{T}_h} \langle \varepsilon_0-\varepsilon_b,  a \nabla ( u-\mathcal{Q}_h^{k-2}u) \cdot
\bm{n}\rangle_{\partial T}\right|\\
 \lesssim & \Big(\sum_{T\in \mathcal{T}_h}\|\varepsilon_0-\varepsilon_b\|_\pT^2\Big)^{\frac{1}{2}} \Big(\sum_{T\in \mathcal{T}_h} \|a\nabla ( u-\mathcal{Q}_h^{k-2}u)\|_\pT^2\Big)^{\frac{1}{2}} \\
\lesssim & \Big(\sum_{T\in \mathcal{T}_h}h_T^{-1}\| u-\mathcal{Q}_h^{k-2}u\|^2_{1,T} + h_T^{1-2\theta}\|u-\mathcal{Q}_h^{k-2}u\|_{2-\theta, T}^2\Big)^{\frac{1}{2}} \Big(\sum_{T\in \mathcal{T}_h}\|\varepsilon_0-\varepsilon_b\|_\pT^2\Big)^{\frac{1}{2}} \\
\lesssim &\  h^{m-1}\|u\|_{m} \3bar \varepsilon_h \3bar.
\end{split}
\end{equation} 

Similarly, from the Cauchy-Schwarz inequality and the trace inequality \eqref{tracein}, the estimate \eqref{term3}  and \eqref{EQ:triplebarnorm}, we obtain
\begin{equation}\label{EQ:J3}
\begin{split}
 &\left| \sum_{T\in \mathcal{T}_h}\langle \varepsilon_n - a\nabla \varepsilon_0\cdot \bn, u-\mathcal{Q}_h^{k-2}u\rangle_{\partial T}\right|\\
\le &\Big(\sum_{T\in \mathcal{T}_h} \|\varepsilon_n-a\nabla \varepsilon_0\cdot \bn \|_\pT^2\Big)^{\frac{1}{2}} \Big(\sum_{T\in \mathcal{T}_h}  \|u-\mathcal{Q}_h^{k-2}u\|_\pT^2\Big)^{\frac{1}{2}}\\
\lesssim &\Big(\sum_{T\in \mathcal{T}_h} \|\varepsilon_n-a\nabla \varepsilon_0\cdot \bn \|_\pT^2\Big)^{\frac{1}{2}}  \Big(\sum_{T\in \mathcal{T}_h} h_T^{-1}\|u-\mathcal{Q}_h^{k-2}u\|_T^2 + h_T^{1-2\theta} \|\nabla(u-\mathcal{Q}_h^{k-2}u)\|_{1-\theta, T}^2 \Big)^{\frac{1}{2}}\\
\lesssim &\  h^{m-1}\|u\|_{m} \3bar \varepsilon_h \3bar.
\end{split}
\end{equation} 

Substituting the estimates \eqref{EQ:J2} and \eqref{EQ:J3} into  \eqref{EQ:April7:001} yields
\begin{equation}\label{aij}
\3bar \varepsilon_h \3bar^2 = |\ell_u(\varepsilon_h)|\lesssim h^{m-1}\|u\|_{m} \3bar \varepsilon_h \3bar,
\end{equation} 
which leads to
\begin{equation}\label{EQ:April7:002}
\3bar \varepsilon_h \3bar  \lesssim h^{m-1}\|u\|_{m}.
\end{equation}
Furthermore, the error equation (\ref{sehv}) yields
$$
b(e_h, \sigma) = \ell_u(\sigma)-s(\varepsilon_h, \sigma), \qquad \forall \sigma\in W_h^0.
$$
It follows from (\ref{aij})-\eqref{EQ:April7:002}, Cauchy-Schwarz inequality, triangle inequality, and \eqref{EQ:triplebarnorm}  that
\begin{equation*}
\begin{split}
|b(e_h, \sigma)| \leq  |\ell_u(\sigma)| + \3bar \varepsilon_h\3bar  \3bar \sigma\3bar \lesssim h^{m-1}\|u\|_{m} \3bar \sigma \3bar\end{split}
\end{equation*}
for all $\sigma\in W_h^0$. Thus, from the \emph{inf-sup} condition (\ref{EQ:inf-sup-condition-01}) we obtain
\begin{equation*}
\begin{split}
\3bar e_h\3bar_{1} &\lesssim h^{m-1}\|u\|_{m},\end{split}
\end{equation*}
which, together with the error estimate (\ref{EQ:April7:002}), completes the proof of the theorem.
\end{proof}

%

For $C^0$-WG elements (i.e., $\sigma_b=\sigma_0$), $\ell_u(\sigma)$ in   (\ref{lu})  is   simplified into 
\begin{equation}\label{sim}
 \ell_u(\sigma) = -\sum_{T\in \mathcal{T}_h} \langle a\nabla \sigma_0\cdot \bn-\sigma_n, u-\mathcal{Q}_h^{k-2}u\rangle_{\partial T},
 \end{equation}
which is applicable to derive the error estimate for the primal variable under ultra-low regularity assumption $u\in H^{1-\theta} (\Omega)$ for $0\leq \theta<\frac12$. The conclusion is stated in the following corollary.
 
 \begin{corollary} \label{ErrorEstimate:4Primal:002}
Let $u$ be the exact solution of the elliptic Cauchy model problem (\ref{model}) and $(u_h, \lambda_h) \in M_{h} \times W_{h}^0$ be the numerical solution arising from PDWG method (\ref{32})-(\ref{2}) with $k\ge 2$. Assume (1) the coefficient tensor $a$ is piecewise constants in $\Omega$ with respect to the finite element partition $\mathcal{T}_h$ which is shape regular as specified in \cite{wy3655}; (2) the exact solution $u$ is sufficiently regular such that $u\in H^{m}(\Omega)$ ($1-\theta\leq m \leq k-1$) with $0\leq \theta<\frac12$. The following error estimates hold true; i.e.,
 \begin{equation*} 
\3bar \lambda_h \3bar+h^{2}\3bar e_h\3bar_{2} 
\lesssim h^{m}\|u\|_{m}.
\end{equation*}
\end{corollary}
\begin{proof}
This proof is easily obtained by following the proof of Theorem \ref{theoestimate} and \eqref{sim}.
\end{proof}
 
\section{Numerical Experiments}\label{Section:numerics}
Extensive numerical results for the PDWG finite element scheme (\ref{32})-(\ref{2}) are reported in this section. The finite element partition $\T_h$ is generated through a successive uniform refinement for a coarse triangulation of the unit square domain $\Omega=(0,1)^2$ by dividing each coarse level triangular element into four congruent sub-triangles by connecting the three mid-points on the edges.
The numerical tests are based on the lowest order $k=2$, where the finite element space for the primal variable is $
M_{h, 2}=\{u_h: u_h|_T \in P_0(T), \forall T\in \mathcal{T}_h\}$,
and the weak finite element space for the dual variable is  
$
W_{h, 2}=\{\lambda_h=\{\lambda_0,\lambda_b, \lambda_n\}:  \lambda_0\in P_2(T), \lambda_b\in P_2(e), \lambda_n\in P_1(e), e\subset\pT, \forall T\in \mathcal{T}_h\}$.


Let $u_h\in M_{h, 2}$ and  $\lambda_h=\{\lambda_0, \lambda_b, \lambda_n\}\in W_{h,2}$ be the numerical solutions arising from the PDWG scheme (\ref{32})-(\ref{2}). 
The primal variable $u_h$ is compared with the exact solution $u$ at the center of  each element  which is known as the nodal point interpolation $I_h u$. The auxiliary variable $\lambda_h$ approximates the true solution $\lambda=0$, and is compared with $Q_h\lambda=0$. The error functions are respectively denoted by
$$
\varepsilon_h=\lambda_h-Q_h\lambda=\{\lambda_0, \lambda_b,  \lambda_n\}, \
\
e_h=u_h-I_h u.
$$
The following $L^2$ norms are used to measure the error functions; i.e.,
$$
  \|e_h\|_0=\Big(\sum_{T\in {\cal
T}_h} \int_T e_h^2 dT\Big)^{\frac{1}{2}}, \quad
  \| \lambda_0\|_0=\Big(\sum_{T\in \mathcal{T}_h}
\int_T \lambda_0^2 dT\Big)^{\frac{1}{2}},$$
$$
\| \lambda_b\|_0=\Big(\sum_{T\in \mathcal{T}_h} h_T
\int_{\partial T} \lambda_b^2 ds\Big)^{\frac{1}{2}},
\| \lambda_n\|_0=\Big(\sum_{T\in \mathcal{T}_h} h_T
\int_{\partial T} \lambda_n^2 ds\Big)^{\frac{1}{2}}.
$$

Tables \ref{NE:TRI:Case1-1}-\ref{NE:TRI:Case1-2} indicate the correctness and reliability of the code for the elliptic Cauchy model problem (\ref{model}). The configuration of the test case is as follows: (1) the exact solution is $u=1$; (2) the coefficient tensor is $a=[1, 0; 0, 1]$; (3) both the Dirichlet and Neumann boundary conditions are set on the horizontal boundary segment $(0, 1)\times 0$ (see Table   \ref{NE:TRI:Case1-1}) or on the vertical boundary segment $0 \times (0, 1)$ (see Table \ref{NE:TRI:Case1-2}). Note that the numerical solution is consistent with the exact solution $u=1$. We observe from Tables \ref{NE:TRI:Case1-1}-\ref{NE:TRI:Case1-2} that the errors are in machine accuracy, especially for relatively coarse grids, which perfectly consists with the developed theory. 
However,  the error seems to deteriorate for the finer meshes which may be caused by the ill-posedness of the elliptic Cauchy problem and/or the poor conditioning of the discrete linear system.

\begin{table}[H]
\begin{center}
\caption{Numerical error with exact solution $u=1$; the coefficient tensor $a=[1, 0; 0, 1]$; Dirichlet and Neumann  on the boundary edge $(0,1)\times 0$.}\label{NE:TRI:Case1-1}
\begin{tabular}{|c|c|c|c|c|}
\hline
$1/h$ & $\|e_h\|_0$   &  $\|\lambda_0\|_{0}$  &     $\|\lambda_b\|_{0}$    & 
$\|\lambda_n\|_{0}$\\
\hline
1&  0E-16&   0.03131E-16    &0.05373E-16  &  0.1412E-16 \\
\hline
2	&   0.05537E-14&   0.01277E-14   &0.01931E-14   &0.1119E-14		
\\
\hline
4	&   0.2243E-14&   0.02680E-14  &0.04223E-14&   0.2141E-14 			
\\
\hline
8& 0.5730E-13&   0.03449E-13&   0.05200E-13&   0.1962E-13			
\\
\hline
16&    0.1824E-12&   0.009643E-12&  0.01422E-12&  0.06788E-12			
\\
\hline
32&     0.1406E-11 &  0.001204E-11&   0.001745E-11 &  0.01216E-11		
\\
\hline
\end{tabular}
\end{center}
\end{table}
\begin{table}[H]
\begin{center}
\caption{Numerical error with exact solution $u=1$; the coefficient tensor $a=[1, 0; 0, 1]$; Dirichlet and Neumann on the boundary edge $0 \times(0,1)$.}\label{NE:TRI:Case1-2}
\begin{tabular}{|c|c|c|c|c|} 
\hline
$1/h$ & $\|e_h\|_0$   &  $\|\lambda_0\|_{0}$  &     $\|\lambda_b\|_{0}$    & 
$\|\lambda_n\|_{0}$\\
\hline
1&    0.2220E-15 &  0.07805E-15 &   0.1246E-15 &   0.5440E-15			
\\
\hline
2	&     0.1407E-14&   0.03347E-14&   0.05395E-14&   0.1939E-14			\\
\hline
4	&      0.2757E-14&  0.01359E-14&   0.01797E-14&   0.2002E-14			\\
\hline
8&    0.2490E-13&   0.009452E-13&   0.01398E-13&   0.1041E-13				
\\
\hline
16&       0.1257E-12&   0.001696E-12&   0.002536E-12&   0.02520E-12			\\
\hline
32&       0.8767E-12&   0.01485E-12&   0.02134E-12&   0.09533E-12			\\
\hline
\end{tabular}
\end{center}
\end{table}

Tables \ref{NE:TRI:Case2-1}-\ref{NE:TRI:Case2-4} demonstrate the numerical performance of the PDWG scheme (\ref{32})-(\ref{2}) when the exact solutions are given by $u = \sin(x) \sin(y)$, $u =\sin(\pi x) \cos(\pi y)$, $u = xy(1-x)(1-y)$ and $u=e^{xy}$, respectively. The coefficient tensor is $a=[1+x^2, 0.25xy; 0.25xy, 1+y^2]$. The boundary conditions are set as follows: (1) Dirichlet on the boundary segments $(0, 1) \times 0$ and $(0, 1) \times 1$; and (2) Neumann on the boundary segments $1 \times  (0, 1)$ and $1 \times (0, 1)$. Note that this is a standard mixed boundary value problem without Cauchy data given on the boundary.  The numerical results show that the convergence rate for $e_h$ in the discrete $L^2$ norm is of an order a bit higher than ${\cal O}(h)$. This group of numerical examples imply the effectiveness of the PDWG algorithm \eqref{32}-\eqref{2} for the classical well-posed problems.

\begin{table}[H]
\begin{center}
\caption{Numerical convergence with exact solution $u=\sin(x) \sin(y)$; the coefficient tensor $a=[1+x^2, 0.25xy; 0.25xy, 1+y^2]$; Dirichlet on the boundary segments $(0, 1) \times 0$ and $(0, 1)\times 1$ and Neumann on the boundary segments
$1\times  (0, 1)$ and $1\times(0, 1)$.}\label{NE:TRI:Case2-1}
\begin{tabular}{|c|c|c|c|c|c|c|c|c|}
\hline
$1/h$ & $\|e_h\|_0$  & order &  $\|\lambda_0\|_{0}$  &  order&   $\|\lambda_b\|_{0}$ &order    & $\|\lambda_n\|_{0}$ & order 
\\
\hline
2&	0.02025&	 &	0.07907&	  &	0.1255 & &	0.7059&	 
\\
\hline
4&	0.008189 &	1.306&	0.01649 &	2.261 &	0.02475&	2.343&	0.1705&	2.050
\\
\hline
8&	0.003478 &	1.235 &	0.003921 &	2.073&	0.005726&	2.112&	0.04248 &	2.005
\\
\hline
16&	0.001604&	1.117&	0.0009725&	2.012&	0.001399 &	2.033 &	0.01063&	1.999
\\
\hline
32&	0.0007741&	1.051&	0.0002434	&1.998&0.00034727 &	2.010 &	0.002659&	1.999 
\\
\hline
\end{tabular}
\end{center}
\end{table}

\begin{table}[H]
\begin{center}
\caption{Numerical convergence with exact solution $u=\sin(\pi x) \cos(\pi y)$; the coefficient tensor $a=[1+x^2, 0.25xy; 0.25xy, 1+y^2]$; Dirichlet on the boundary segments $(0, 1) \times 0$ and $(0, 1)\times 1$ and Neumann on the boundary segments
$1\times  (0, 1)$ and $1\times(0, 1)$.}\label{NE:TRI:Case2-2}
\begin{tabular}{|c|c|c|c|c|c|c|c|c|}
\hline
$1/h$ & $\|e_h\|_0$  & order &  $\|\lambda_0\|_{0}$  &  order&   $\|\lambda_b\|_{0}$ &order    & $\|\lambda_n\|_{0}$ & order\\
\hline
1&	0.1983 	&&	1.775	&&	3.035 	&&5.942 	&
\\
\hline
2	&0.07463&	1.410 &	0.2382	&2.897 &0.3974	&2.933 &1.147 &	2.373
\\
\hline
4	&0.02809 &	1.410 &	0.04732	&2.332&	0.07001&	2.505&	0.3657	&1.650
\\
\hline
8&	0.01183 &	1.247 &	0.01125 &	2.072 &	0.01625 	&2.107	&0.09009 &2.021 
\\
\hline
16&	0.004935 &	1.262&	0.002793 &2.010 &	0.003997&	2.024&	0.02209 &	2.028
\\
\hline
32&	0.002181 &	1.178 &	0.0006993&	1.998&	0.0009954&	2.006 	&0.005478 & 2.012
\\
\hline
\end{tabular}
\end{center}
\end{table}

\begin{table}[H]
\begin{center}
\caption{Numerical convergence with exact solution $u=xy(1-x)(1-y)$; the coefficient tensor $a=[1+x^2, 0.25xy; 0.25xy, 1+y^2]$; Dirichlet on the boundary segments $(0, 1) \times 0$ and $(0, 1)\times 1$ and Neumann on the boundary segments
$1\times  (0, 1)$ and $1\times(0, 1)$.}\label{NE:TRI:Case2-3}
\begin{tabular}{|c|c|c|c|c|c|c|c|c|}
\hline
$1/h$ & $\|e_h\|_0$  & order &  $\|\lambda_0\|_{0}$  &  order&   $\|\lambda_b\|_{0}$ &order    & $\|\lambda_n\|_{0}$ & order\\
\hline
1&	0.03376&&	0.1206 	&&	0.2025	&&0.1335&	
\\
\hline
2	&0.008454 	&1.998 &	0.02765 &	2.125&	0.03792 &	2.417&	0.04527&	1.560
\\
\hline
4&	0.002336 &	1.856&	0.006620 &	2.062 &	0.009235 &	2.038&	0.01260 & 	1.845
\\
\hline
8&	0.0008282&	1.496 &	0.001651 &	2.003 &	0.002325 	&1.990 &	0.003207 &	1.975 
\\
\hline
16&	0.0003441 &	1.267 &	0.0004133 &	1.998 &	0.0005838 &	1.994 &	0.0007990 &	2.005 
\\
\hline
32&	0.0001560 &	1.142 	&0.0001034 &1.999&	0.0001462 &	1.998	&0.0001979&	2.014\\
\hline
\end{tabular}
\end{center}
\end{table}

 \begin{table}[H]
\begin{center}
\caption{Numerical convergence with exact solution $u=e^{xy}$; the coefficient tensor $a=[1+x^2, 0.25xy; 0.25xy, 1+y^2]$; Dirichlet on the boundary segments $(0, 1) \times 0$ and $(0, 1)\times 1$ and Neumann on the boundary segments
$1\times  (0, 1)$ and $1\times(0, 1)$.}\label{NE:TRI:Case2-4}
\begin{tabular}{|c|c|c|c|c|c|c|c|c|}
\hline
$1/h$ & $\|e_h\|_0$  & order &  $\|\lambda_0\|_{0}$  &  order&   $\|\lambda_b\|_{0}$ &order    & $\|\lambda_n\|_{0}$ & order
\\
\hline  
2&	0.05293 &	 &	0.1790 &  &	0.2650 &	 	&1.604 	& 
\\
\hline 
4&	0.02334 &	1.181 &	0.04208 &	2.089 &	0.06047	&2.132 	&0.4208 &	1.931
\\
\hline 
8&	0.009735 &	1.262&	0.01057&	1.993 &	0.01506 &	2.005&	0.1080 &	1.962 
\\
\hline 
16&	0.004345 &	1.164 &	0.002679 &	1.980 &	0.003805 &	1.985 &	0.02742 &	1.978 
\\
\hline
32	&0.002050 &1.084 &	0.0006762 &1.986	&0.0009587 	&1.989&	0.006912&1.988
\\
\hline
\end{tabular}
\end{center}
\end{table}

Tables \ref{NE:TRI:Case3-1}-\ref{NE:TRI:Case3-4} demonstrate the performance of the PDWG algorithm (\ref{32})-(\ref{2}) when the exact solutions are given by $u = \sin(x) \sin(y)$, $u =\cos(x) \cos(y)$, $u=xy(1-x)(1-y)$ and $u=e^{xy}$, respectively. The setting of boundary conditions is as follows: (1) both Dirichlet and Neumann on the boundary segments $(0, 1) \times 0$ and $(0, 1) \times 1$; and (2) Dirichlet on the boundary segment $0\times  (0, 1)$.  The coefficient tensor is $a=[1,0; 0,1]$. The numerical results in Tables \ref{NE:TRI:Case3-1}-\ref{NE:TRI:Case3-4} show that the convergence order for the primal variable is higher than ${\cal O}(h)$.

\begin{table}[H]
\begin{center}
\caption{Numerical convergence with exact solution $u=\sin(x)\sin(y)$; the coefficient tensor $a=[1, 0; 0, 1]$; Dirichlet and Neumann on the boundary segments $(0, 1) \times 0$ and $(0, 1) \times 1$, and Dirichlet on the boundary segment $0\times (0, 1)$.}\label{NE:TRI:Case3-1}
\begin{tabular}{|c|c|c|c|c|c|c|c|c|}
\hline
$1/h$ & $\|e_h\|_0$  & order &  $\|\lambda_0\|_{0}$  &  order&   $\|\lambda_b\|_{0}$ &order    & $\|\lambda_n\|_{0}$ & order\\
\hline
1 &0.07102 &&0.3979	&&	0.6416 	&&	2.777	&
\\
\hline
2&	0.03856 &	0.8810 &	0.06642&	2.583&	0.1032 	&2.636 &	0.4885&	2.507 
\\
\hline
4	&0.01390 &	1.473	&0.02039 &	1.703 	&0.03087&	1.741 &	0.1443 &	1.759 
\\
\hline
8&		0.003625&		1.939 &		0.006680 &		1.610 &		0.009868	&	1.645 &		0.04291	&	1.750 
\\
\hline
16&		0.001311&		1.467 &		0.002131 	&	1.648 &		0.003087 &		1.676 &		0.01235 	&	1.796 
\\
\hline
32&	 0.0004754& 1.463 	&	0.0007006 &	1.605 &	0.001005& 1.620	&	0.004066	&	  1.603 
\\
\hline
\end{tabular}
\end{center}
\end{table}

\begin{table}[H]
\begin{center}
\caption{Numerical convergence with exact solution $u=\cos(x)\cos(y)$; the coefficient tensor $a=[1, 0; 0, 1]$; Dirichlet and Neumann on the boundary segments $(0, 1) \times 0$ and $(0, 1) \times 1$, and Dirichlet on the boundary segment $0\times (0, 1)$.}\label{NE:TRI:Case3-2}
\begin{tabular}{|c|c|c|c|c|c|c|c|c|}
\hline
$1/h$ & $\|e_h\|_0$  & order &  $\|\lambda_0\|_{0}$  &  order&   $\|\lambda_b\|_{0}$ &order    & $\|\lambda_n\|_{0}$ & order \\
\hline
2&	0.02161&	 	&0.09323 & 	&0.138 &  &	0.4247	& 
\\
\hline
4&	0.009932	&1.121 &	0.02531	&1.881 &0.03693 &	1.904&	0.1116&	1.928 
\\
\hline
8&	0.003352	&1.567 	&0.006692&	1.919 &	0.009652 &	1.936 &	0.02920 &	1.934
\\
\hline
16&	0.001396&	1.264	&0.001836 &	1.866&	0.002630&	1.876&	0.007239	&2.012 
\\
\hline
32	&0.0006223 &1.165 &	0.0005124&	1.841 &	0.0007307 &	1.847 &	0.001986 &	1.866
\\
\hline
\end{tabular}
\end{center}
\end{table}

 \begin{table}[H]
\begin{center}
\caption{Numerical convergence with exact solution $u=xy(1-x)(1-y)$; the coefficient tensor $a=[1, 0; 0, 1]$; Dirichlet and Neumann on the boundary segments $(0, 1) \times 0$ and $(0, 1) \times 1$, and Dirichlet on the boundary segment $0\times (0, 1)$.}\label{NE:TRI:Case3-3}
\begin{tabular}{|c|c|c|c|c|c|c|c|c|}
\hline
$1/h$ & $\|e_h\|_0$  & order &  $\|\lambda_0\|_{0}$  &  order&   $\|\lambda_b\|_{0}$ &order    & $\|\lambda_n\|_{0}$ & order\\
\hline
1&	0.01042	&&	0.07688 &&	0.1406 &&	0.2900 &
\\
\hline
2&	0.004678&	1.155 &	0.02185 &	1.815&	0.02976 &	2.240 &	0.03095 &3.228
\\
\hline
4&	0.001630 	&1.521	&0.005182 	&2.076 	&0.007245 &2.038&	0.005847 &	2.404 
\\
\hline
8&	0.0005178&	1.655&	0.001261&	2.039 &	0.001780 &2.025 &	0.001644&	1.831
\\
\hline
16&	0.0002162&	1.260 &	0.0003168&	1.993&	0.0004484 	&1.989&	0.0004812	&1.772 
\\
\hline
32&	9.508E-05&	1.185&	8.028E-05&	1.980 	&0.0001137 &1.979 	&0.0001483 & 1.698
\\
\hline
\end{tabular}
\end{center}
\end{table}

 \begin{table}[H]
\begin{center}
\caption{Numerical convergence with exact solution $u=e^{xy}$; the coefficient tensor $a=[1, 0; 0, 1]$; Dirichlet and Neumann on the boundary segments $(0, 1) \times 0$ and $(0, 1) \times 1$, and Dirichlet on the boundary segment $0\times (0, 1)$.}\label{NE:TRI:Case3-4}
\begin{tabular}{|c|c|c|c|c|c|c|c|c|}
\hline
$1/h$ & $\|e_h\|_0$  & order &  $\|\lambda_0\|_{0}$  &  order&   $\|\lambda_b\|_{0}$ &order    & $\|\lambda_n\|_{0}$ & order\\
\hline
1	&0.1067 &&0.6226	&&	0.8525 	&&	4.572 	& 
\\
\hline
2&	0.05165 &	1.047 &	0.1483 &	2.070&	0.2128 &	2.002&	0.9598 &	2.252 
\\
\hline
4	&0.01646 &1.650 &	0.04342 &	1.772 	&0.06371&	1.740 &	0.2797	&1.779 
\\
\hline
8&	0.005303 &	1.634&	0.01198 &	1.858 &	0.01749 &	1.865&	0.07875&	1.828 
\\
\hline
16&	0.002344 &	1.178&	0.002741&	2.128&	0.003944 &	2.149	&0.01915	&2.040
\\
\hline
32	&0.0009561&1.294&	0.0007732&	1.826&	0.001105 	&1.836	&0.005242 &	1.869
\\
\hline
\end{tabular}
\end{center}
\end{table}

Tables \ref{NE:TRI:Case4-1}-\ref{NE:TRI:Case4-2} illustrate the numerical performance of the PDWG algorithm when both Dirichlet and Neumann conditions are given on the boundary segments $(0, 1) \times 0$ and $(0, 1) \times 1$. The exact solutions are taken by $u = \sin(\pi x) \sin(\pi y)$ and $u = xy(1-x)(1-y)$, respectively; and the coefficient tensor is $a=[1,0; 0,1]$. We can observe from the numerical results that the convergence rate for $e_h$ in the discrete $L^2$ norm is of an order a bit higher than ${\cal O}(h)$.

\begin{table}[H]
\begin{center}
\caption{Numerical convergence with exact solution $u=\sin(\pi x)\sin(\pi y)$; the coefficient tensor $a=[1, 0; 0, 1]$; Dirichlet and Neumann on the boundary segments $(0, 1) \times 0$ and $(0, 1) \times 1$.}\label{NE:TRI:Case4-1}
\begin{tabular}{|c|c|c|c|c|c|c|c|c|}
\hline
$1/h$ & $\|e_h\|_0$  & order &  $\|\lambda_0\|_{0}$  &  order&   $\|\lambda_b\|_{0}$ &order    & $\|\lambda_n\|_{0}$ & order\\
\hline
1&	0.1281 	&&	0.8807 	&&	1.501&&		0.3968 &	
\\
\hline
2	&0.04934&	1.377 &	0.2939&	1.583 	&0.4002 &	1.907	&0.3624 &0.1307
\\
\hline
4&	0.02366 &	1.060 	&0.07421 	&1.986&	0.1037&	1.949 &	0.08287	&2.129
\\
\hline
8&	0.01312 &	0.8503 	&0.01867&1.991 	&0.02633	&1.977 &	0.02670	&1.634 
\\
\hline
16&	0.006379	&1.041&	0.004740 &	1.977 &0.006706 &1.973 &	0.008679 &	1.621 
\\
\hline
32&	0.003017&	1.080 &	0.001215 &	1.964 &	0.001720 &1.963&	0.002968 &	1.548 
\\
\hline
\end{tabular}
\end{center}
\end{table}

\begin{table}[H]
\begin{center}
\caption{Numerical convergence with exact solution $u=xy(1-x)(1-y)$; the coefficient tensor $a=[1, 0; 0, 1]$; Dirichlet and Neumann on the boundary segments $(0, 1) \times 0$ and $(0, 1) \times 1$.}\label{NE:TRI:Case4-2}
\begin{tabular}{|c|c|c|c|c|c|c|c|c|}
\hline
$1/h$ & $\|e_h\|_0$  & order &  $\|\lambda_0\|_{0}$  &  order&   $\|\lambda_b\|_{0}$ &order    & $\|\lambda_n\|_{0}$ & order\\
\hline

2& 	0.002822 &   &  0.02069&   & 	0.02818	&  	& 0.02300 &  
\\
\hline
4& 	0.001610& 	0.8103 & 	0.005023 	& 2.042 	& 0.007021	& 2.005& 	0.005487& 	2.068
\\
\hline
8& 	0.0008414 & 	0.9358 & 	0.001246 	& 2.011	& 0.001758& 	1.998& 	0.001748 & 	1.650 
\\
\hline
16	& 0.0003811 & 1.142 & 	0.0003160 & 	1.979 	& 0.0004471& 	1.975 & 	0.0005608 & 1.640 
\\
\hline
32& 	0.0001708& 	1.158 	& 8.086E-05&  1.967	& 0.0001145	& 1.966	& 0.0001848&  1.602 
\\
\hline
\end{tabular}
\end{center}
\end{table}

Table \ref{NE:TRI:Case6-1} presents the numerical performance of the PDWG algorithm in the following configuration: (1) the exact solution is $u = e^{xy}$; (2) the coefficient tensor is  $a=[1+x^2, 0.25xy; 0.25xy, 1+y^2]$; (3) both Dirichlet and Neumann conditions are given on the boundary segments $(0, 1) \times 0$ and $(0, 1) \times 1$. We can observe from Table \ref{NE:TRI:Case6-1} that the convergence rate for the numerical solution $u_h$ in the discrete $L^2$ norm is of an order a little higher than ${\cal O}(h)$.

\begin{table}[H]
\begin{center}
\caption{Numerical convergence with exact solution $u=e^{xy}$; the coefficient tensor $a=[1+x^2, 0.25xy; 0.25xy, 1+y^2]$; Dirichlet and Neumann on the boundary segments $(0, 1) \times 0$ and $(0, 1) \times 1$.}\label{NE:TRI:Case6-1}
\begin{tabular}{|c|c|c|c|c|c|c|c|c|}
\hline
$1/h$ & $\|e_h\|_0$  & order &  $\|\lambda_0\|_{0}$  &  order&   $\|\lambda_b\|_{0}$ &order    & $\|\lambda_n\|_{0}$ & order\\
\hline
2&	0.09498 & &	0.1600 &	  &	0.2125&	  &	1.629 &	 
\\
\hline
4&	0.08063 &	0.2362 &	0.03763 &	2.088 &	0.05258 &	2.015&	0.4707 &	1.791
\\
\hline
8&	0.03599 &	1.164 &	0.01069 &	1.821 &	0.01550 &	1.762 &	0.1461&	1.688 
\\
\hline
16&	0.01465 	&1.297 &	0.002853	&1.900 &0.004152 &	1.900 &	0.04208 &	1.796
\\
\hline
32&	0.006831&	1.100 &	0.0008062 &	1.823 &	0.001161 	&1.838 	&0.01166&	1.852
\\
\hline
\end{tabular}
\end{center}
\end{table}

Tables \ref{NE:TRI:Case7-1}-\ref{NE:TRI:Case7-3} demonstrate the performance of the PDWG scheme for the numerical tests in the following configuration: (1) the coefficient tensor is  $a=[1+x^2, 0.25xy; 0.25xy, 1+y^2]$; (2) the boundary conditions are set by Dirichlet and Neumann on the boundary segments $(0, 1) \times 0$ and $(0, 1) \times 1$ and Neumann on the boundary segment 
$0 \times (0, 1)$; (3) the exact solutions are given by $u = \sin(x) \cos(y)$ and $u=e^{xy}$, respectively. Tables \ref{NE:TRI:Case7-1}-\ref{NE:TRI:Case7-3} indicate that the convergence order for the PDWG  solution $u_h$ in the discrete $L^2$ norm is of an order higher than the order ${\cal O}(h)$.

\begin{table}[H]
\begin{center}
\caption{Numerical convergence with exact solution $u=\sin(x)\cos(y)$; the coefficient tensor $a=[1+x^2, 0.25xy; 0.25xy, 1+y^2]$; Dirichlet and Neumann on the boundary segments $(0, 1) \times 0$ and $(0, 1) \times 1$, and Neumann 
on the boundary segment $0\times (0, 1)$.}\label{NE:TRI:Case7-1}
\begin{tabular}{|c|c|c|c|c|c|c|c|c|}
\hline
$1/h$ & $\|e_h\|_0$  & order &  $\|\lambda_0\|_{0}$  &  order&   $\|\lambda_b\|_{0}$ &order    & $\|\lambda_n\|_{0}$ & order
\\
\hline
1&	0.05811 &&	1.050 	&&1.993&&0.8839 	&
\\
\hline
2&	0.02234 &	1.379 &	0.2415 	&2.121 	&0.4034 	&2.304 &	0.1804 &2.293
\\
\hline
4&	0.009729	&1.199 &	0.06850 	&1.818	&0.1056	&1.934&	0.05650 &	1.675
\\
\hline
8&	0.004034	&1.270 &	0.02069 &	1.727 	&0.03058&	1.788 &	0.02364 &	1.257 
\\
\hline
16&	0.001292 &	1.642 &	0.005941	&1.800 &0.008589 &	1.832&	0.008732&	1.437
\\
\hline
32	&0.0005206 &	1.311 &	0.001395 &	2.090 &	0.001994 &	2.107 &	0.002616 &	1.739 
\\
\hline
\end{tabular}
\end{center}
\end{table}

\begin{table}[H]
\begin{center}
\caption{Numerical convergence with exact solution $u=e^{xy}$; the coefficient tensor $a=[1+x^2, 0.25xy; 0.25xy, 1+y^2]$; Dirichlet and Neumann on the boundary segments $(0, 1) \times 0$ and $(0, 1) \times 1$, and Neumann 
on the boundary segment $0\times (0, 1)$.}\label{NE:TRI:Case7-3}
\begin{tabular}{|c|c|c|c|c|c|c|c|c|}
\hline
$1/h$ & $\|e_h\|_0$  & order &  $\|\lambda_0\|_{0}$  &  order&   $\|\lambda_b\|_{0}$ &order    & $\|\lambda_n\|_{0}$ & order\\
\hline
1&	0.1009		&&2.097	&&	3.548 	&&	3.380 &	\\
\hline
2&	0.05792&	0.8002&	0.4685 &	2.162 &	0.7333 &	2.275&	1.300 &	1.379
\\
\hline
4	&0.02899	&0.9985 &	0.1087&	2.108 	&0.1647  &	2.154	&0.5187	&1.326
\\
\hline
8&	0.01315&	1.140 	&0.01971 &	2.463	&0.02922 &	2.495	&0.1603&	1.694 
\\
\hline
16&	0.004980 &	1.401&	0.005810 &	1.762 &	0.008517 &	1.779&	0.04668&	1.780
\\
\hline
32&	0.002143 &	1.216 &	0.001291&	2.170 &	0.001858 &	2.196 &	0.01099 &	2.086 
\\
\hline
\end{tabular}
\end{center}
\end{table}

Tables \ref{NE:TRI:Case8-1}-\ref{NE:TRI:Case8-2} demonstrate the performance of the PDWG algorithm  for the exact solutions $u = xy(1-x)(1-y)$ and $u=e^{xy}$, respectively. The  boundary conditions are set as follows: (1) Dirichlet and Neumann on the boundary segments $(0, 1) \times 0$ and $(0, 1) \times 1$, and (2) Neumann on the boundary segment $0\times  (0, 1)$. The coefficient tensor is $a=[1, 0; 0, 1]$. The numerical results in Tables \ref{NE:TRI:Case8-1}-\ref{NE:TRI:Case8-2} illustrate that the convergence rate for $e_h$ in the discrete $L^2$ norm is higher than the order ${\cal O}(h)$.

\begin{table}[H]
\begin{center}
\caption{Numerical convergence with exact solution $u = xy(1-x)(1-y)$; the coefficient tensor $a=[1, 0; 0, 1]$; Dirichlet and Neumann on the boundary segments $(0, 1) \times 0$ and $(0, 1) \times 1$, and Dirichlet
on the boundary segment $0\times (0, 1)$.}\label{NE:TRI:Case8-1}
\begin{tabular}{|c|c|c|c|c|c|c|c|c|}
\hline
$1/h$ & $\|e_h\|_0$  & order &  $\|\lambda_0\|_{0}$  &  order&   $\|\lambda_b\|_{0}$ &order    & $\|\lambda_n\|_{0}$ & order
\\
\hline
1	&0.01042	&&	0.07688 	&&	0.1406 &&	0.2900 	&
\\
\hline
2&	0.004678 	&1.155 &	0.02185 &	1.815&	0.02976 &	2.240 &	0.03095 &	3.228
\\
\hline
4&	0.001630 &	1.521&	0.005182 &	2.076 &	0.007245 &	2.038 &0.005847 &	2.404 
\\
\hline
8&	0.0005178&	1.655	&0.001261	&2.040	&0.001780	&2.025 &	0.001644	&1.831
\\
\hline
16&	0.0002162	&1.260 &0.0003168&	1.993&	0.0004484 &	1.989	&0.0004812&	1.772 
\\
\hline
32&	9.508E-05&	1.185 &	8.028E-05&	1.980 &	0.0001137 &1.980	&0.0001483 &1.698
\\
\hline
\end{tabular}
\end{center}
\end{table}

\begin{table}[H]
\begin{center}
\caption{Numerical convergence with exact solution $u = e^{xy}$; the coefficient tensor $a=[1, 0; 0, 1]$; Dirichlet and Neumann on the boundary segments $(0, 1) \times 0$ and $(0, 1) \times 1$, and Dirichlet
on the boundary segment $0\times (0, 1)$.}\label{NE:TRI:Case8-2}
\begin{tabular}{|c|c|c|c|c|c|c|c|c|}
\hline
$1/h$ & $\|e_h\|_0$  & order &  $\|\lambda_0\|_{0}$  &  order&   $\|\lambda_b\|_{0}$ &order    & $\|\lambda_n\|_{0}$ & order
\\
\hline
1&	0.1067&&	0.6226&&		0.8525 &&		4.572 &
\\
\hline
2&	0.05165 &	1.047 &	0.1483 &	2.070 &	0.2128 &	2.002&	0.9598 &	2.252 
\\
\hline
4&	0.01646	&1.650 	&0.04342 &	1.772 &	0.06371&	1.740 &	0.27965 &	1.779 
\\
\hline
8&	0.005303 &	1.634&	0.01197662	&1.858 	&0.01749 &1.865&	0.07875&	1.828 
\\
\hline
16	&0.002344 &	1.178&	0.002741	&2.128 &	0.003944 &	2.149 &	0.01915 &	2.040 
\\
\hline
32&	0.0009561&	1.294&	0.0007732&	1.826 &	0.001105 &	1.836 &	0.005242 	&1.869
\\
\hline
\end{tabular}
\end{center}
\end{table}

Table \ref{NE:TRI:Case5-1} demonstrate the performance of the PDWG method  when the exact solutions are given by $u_1=\sin(x)\cos(y)$, $u_2=\sin(\pi x)\sin(\pi y)$, $u_3=\sin(x)\sin(y)$ and $u_4=e^{xy}$ respectively. The boundary conditions are set as follows: (1) Dirichlet and Neumann conditions on the boundary segments $(0, 1) \times 0$ and $(0, 1) \times 1$; (2) Dirichlet on the boundary segment $0 \times (0,1)$; and (3) Neumann on the boundary segment $1 \times (0,1)$. The coefficient tensor is chosen by $a=[1+x^2, 0.25xy; 0.25xy, 1+y^2]$. We observe from Table \ref{NE:TRI:Case5-1} that the convergence order for the error $e_h$ in the discrete $L^2$ norm is of an order higher than ${\cal O}(h)$.

\begin{table}[H]
\begin{center}
\caption{Numerical convergence with exact solutions $u_1=\sin(x)\cos(y)$, $u_2=\sin(\pi x)\sin(\pi y)$, $u_3=\sin(x)\sin(y)$ and $u_4=e^{xy}$; the coefficient tensor $a=[1+x^2, 0.25xy; 0.25xy, 1+y^2]$; Dirichlet and Neumann on the boundary segments $(0, 1) \times 0$ and $(0, 1) \times 1$, Dirichlet on the boundary segment $0 \times (0,1)$ and Neumann on the boundary segment $1 \times (0,1)$.}\label{NE:TRI:Case5-1}
\begin{tabular}{|c|c|c|c|c|c|c|c|c|}
\hline
$1/h$ & $\|e_h\|_0$($u_1$)  & order &  $\|e_h\|_0$($u_2$)  &  order&  $\|e_h\|_0$($u_3$)  &  order&  $\|e_h\|_0$($u_4$)  &  order  
\\
\hline
2	&0.01916&	  &	0.06362&	 &	0.02025&	 &	0.06538 &	
\\
\hline
4&	0.006850 &1.484&	0.02385 &	1.415 	&0.008189 &	1.306 &0.02519 	&1.376 
\\
\hline
8&	0.002455 &	1.480 &	0.008741 &	1.448 &	0.003478 &1.235 &	0.01014 &	1.312 
\\
\hline
16&	0.0008979 &	1.451 	&0.003645 &	1.262&	0.001604&	1.117&	0.004450 	&1.189
\\
\hline
32&	0.0003603&	1.317 &	0.001663 &	1.132&	0.0007741	&1.051	&0.002072 	&1.103
\\
\hline
\end{tabular}
\end{center}
\end{table}


\begin{thebibliography}{99}
\bibitem{s5} {\sc G. Alessandrini}, {\em Stable determination of a crack from boundary measurements}, Proc. Roy. Soc., Edinburgh Sect. A., vol. 123, pp. 497-516, 1993.

\bibitem{s13} {\sc G. Alessandrini, L. Rondi, E. Rosset and S. Vessella}, {\em  The stability for the Cauchy problem for elliptic equations}, Inverse Problems, vol. 25, pp. 1-47, 2009.

 \bibitem{Andrieux} {\sc S. Andrieux, T. N. Baranger and A. Ben Abda},
 {\em Solving Cauchy problems by minimizing an energy-like functional},
 Inverse Problems,  Volume 22,  Number 1, pp. 15-33, 2006.

\bibitem{s1} {\sc M. Bai }, {\em  Application of BEM-based acoustic holography to radiation analysis of sound sources with arbitrarily shaped geometries}, J. Acoust. Soc. Am., vol. 92, pp. 533-549,1992.

 \bibitem{Belgacem} {\sc F. Belgacem}, {\em Why is the Cauchy problem severely ill-posed?}, Inverse Problems, vol. 23, pp. 823-836, 2007.
 
 \bibitem{r2} {\sc F. Belgacem and H. Fekih}, {\em  On Cauchy's problem: I. A variational Steklov-Poincare theory}, Inverse Problems, vol. 21, pp. 1915-1936, 2005.

\bibitem{k6} {\sc A. Bjorck, E. Grimme and P. Dooren}, {\em An implicit shift bidiagonalization algorithm for ill-posed systems}, BIT, vol. 34, pp. 510-534, 1994.
 
\bibitem{Bourgeois} {\sc L. Bourgeois}, {\em A mixed formulation of
quasi-reversibility to solve the Cauchy problem for Laplace's
equation}, Inverse Problems, vol. 21, pp. 1087-1104, 2005.

 \bibitem{s6} {\sc A. Bukhgeim, J. Cheng and M. Yamamoto}, {\em Stability for an inverse boundary problem of determining a part of a boundary}, Inverse Problems, vol. 15, pp. 1021-1032, 1999.

\bibitem{s7} {\sc A. Bukhgeim, J. Cheng and M. Yamamoto}, {\em  On a sharp estimate in a non-destructive testing: determination of unknown boundaries}, K. Miya, M. Yamamoto, Xuan Hung Nguyen (Eds.), Applied electromagnetism mechanics, JSAEM, pp. 64-75, 1998.

\bibitem{ErikBurman-EllipticCauchy} {\sc E. Burman},
{\em Error estimates for stabilized finite element methods applied
to ill-posed problems}, C. R. Acad. Sci. Paris, Ser., vol. I 352,
pp. 655-659, 2014. http://dx.doi.org/10.1016/j.crma.2014.06.008
\bibitem{Burman} {\sc E. Burman}, {\em Stabilized finite element methods for
nonsymmetric, noncoercive, and ill-posed problems. Part I: Elliptic
equations}, SIAM J. Sci. Comput., vol. 35, pp. 2752-2780, 2013.

 \bibitem{cao} {\sc W. Cao and C. Wang}, {\em New Primal-Dual Weak Galerkin Finite Element Methods for Convection-Diffusion Problems}. arXiv:2001.06847. 

\bibitem{r10} {\sc A. Chakib and A. Nachaoui}, {\em Convergence analysis for finite element approximation to an inverse Cauchy problem}, Inverse Problems, vol. 22, pp. 1191-1206, 2006.

\bibitem{s18} {\sc W. Chen and Z. Fu}, {\em Boundary particle method for inverse Cauchy problem of inhomogeneous inhomogeneous Helmholtz equations}, J. Mar. Sci. Technol., vol. 17, pp. 157-163, 2009.

\bibitem{s8} {\sc J. Cheng, S. Prossdorf and M. Yamamoto}, {\em Local estimation for an integral equation of first kind with analytic kernel}, J. Inverse Ill-Posed Probl., vol. 6, pp. 115-126, 1998.

\bibitem{s9} {\sc J. Cheng and M. Yamamoto}, {\em  Local stability of a linearized inverse problem in detecting steel reinforcement bars}, in: Proc. Int. Conf. Inverse Problems and Applications, Quezon City, MatimyIas Mat., vol. 21, pp. 18-33, 1998.

\bibitem{r6} {\sc A. Cimetiere, F. Delvare, M. Jaoua and F. Pons}, {\em Solution of the Cauchy problem using iterated Tikhonov regularization}, Inverse Problems, vol. 17, pp. 553-570, 2001.

\bibitem{s11} {\sc P. Colli-Franzone, L. Guerri, S. Tentoni, C. Viganotti, S. Baruffi, S. Spaggiari and B. Taccardi}, {\em A mathematical procedure for solving the inverse potential problem of electrocardiography}. Analysis of the time-space accuracy from in vitro experimental data, Math. Biosci., vol. 77, pp. 353-396, 1985.

\bibitem{k10} {\sc L. Elden}, {\em Numerical solution of the sideways heat equation by difference approximation in time}, Inverse Problems, vol. 11, pp. 913-923, 1995.

\bibitem{k11} {\sc L. Elden and F. Berntsson}, {\em Spectral and wavelet methods for solving an inverse heat conduction problem}, International Symposium on Inverse Problems in Engineering Mechanics, Nagano, Japan, 1998.

\bibitem{Falk-Monk} {\sc R. Falk and P. Monk}, {\em Logarithmic convexity for discrete harmonic functions and the approximation of the Cauchy problem for
Poissons equation}, Math. Comp., vol. 47, pp. 135-149, 1986.
 
\bibitem{Gilbarg-Trudinger}
{\sc David Gilbarg and Neil S. Trudinger}. {\em Elliptic Partial
Differential Equations of Second Order}. Springer-Verlag, Berlin,
second edition, 1983.
 
\bibitem{h3} {\sc J. Hadamard }, {\em Surles fonctions entieres}, Bull. Soc. Math. France, vol. 24, pp. 94-96, 1896.

\bibitem{e17} {\sc J. Hadamard }, {\em Surles problems aux derives partielles et leur signification physique}, Princeton Univ. Bull., vol. 13, pp. 49-52, 1902.

 \bibitem{h2} {\sc J. Hadamard}, {\em Lectures on Cauchy's problem in linear partial differential equations}, New Haven: Yale University Press; London: Humphrey Milford; Oxford: University Press. VIII u. 316 S., 1923.

 \bibitem{h1} {\sc  J. Hadamard}, {\em  Lectures on Cauchy's Problem in Linear Partial Differential Equation}, Dover, New York, 1953.
 
\bibitem{s2} {\sc W. Hall and X. Mao }, {\em Boundary element investigation of irregular frequencies in electro magnetic scattering}, Eng. Anal. Bound. Elem., vol. 16, pp. 245-252, 1995.

\bibitem{s14} {\sc T. Hrycak and V. Isakov}, {\em Increased stability in the continuation of solutions to the Helmholtz equation}, Inverse Problems, vol. 20, pp. 697-712, 2004.
 
\bibitem{s15} {\sc V. Isakov}, {\em  Inverse Problems for Partial Differential Equations}, Springer-Verlag, New York, 1998.

 \bibitem{s20} {\sc B. Jin and Y. Zheng}, {\em Boundary knot method for the Cauchy problem associated with the inhomogeneous Helmholtz equation}, Eng. Anal. Bound. Elem., vol. 29, pp. 925-935, 2005.

 
\bibitem{s21} {\sc B. Jin and Y. Zheng}, {\em A meshless method for some inverse problems associated with the inhomogeneous Helmholtz equation}, Comput. Methods Appl. Mech. Eng., vol. 195, pp. 2270-2288, 2006.

\bibitem{s3} {\sc B. Kim and J. Ih}, {\em On the reconstruction of the vibroacoustic field over the surface enclosing an interior space using the boundary element method}, J. Acoust. Soc. Am., vol. 100, pp. 3003-3016, 1996.

 \bibitem{s16} {\sc J. Lee and J. Yoon}, {\em A numerical method for Cauchy problem using singular value decomposition}, Comm. Korean Math. Soc., vol. 16, pp. 487-508, 2001.

\bibitem{Leitao} {\sc A. Leit\~{a}o and M. Marques Alves}, {\em On level set type methods for elliptic Cauchy problems}, Inverse Problems 23 (2007), pp. 2207-2222. doi:10.1088/0266-5611/23/5/023.


\bibitem{tr2} {\sc Dan Li,  C. Wang and J. Wang}, {\em  Primal-Dual Weak Galerkin Finite Element Methods for Linear Convection Equations in Non-Divergence Form}. arXiv: 1910.14073.
 

 
 \bibitem{r8} {\sc L. Marin}, {\em  Convergence analysis for finite element approximation to an inverse Cauchy problem}, International Journal of Solids and Structures, vol. 42, pp. 4338-4351, 2005.
 
 \bibitem{r9} {\sc L. Marin}, {\em The minimal error method for the Cauchy next term problem in linear elasticity. Numerical implementation for two-dimensional homogeneous isotropic linear elasticity}, International Journal of Solids and Structures, vol. 46,  pp. 957-974, 2005.

\bibitem{s4} {\sc L. Marin, L. Elliott, P. Heggs, D. Ingham, D. Lesnic and X. Wen}, {\em  An alternating iterative algorithm for the Cauchy problem associated to the Helmholtz equation}, Comput. Methods Appl. Mech. Eng., vol. 192, pp. 709-722, 2003.

\bibitem{s19} {\sc L. Marin, L. Elliott, P. Heggs, D. Ingham, D. Lesnic and X. Wen}, {\em Conjugate gradient-boundary element solution to the Cauchy problem for the Helmholtz-type equations}, Comput. Mech., vol. 31, pp. 367-377, 2003.

 
\bibitem{s22} {\sc L. Marin and D. Lesnic}, {\em The Method of fundamental solutions for the Cauchy problem associated with two-dimensional the Helmholtz-type equations}, Comput. Struct., vol. 83, pp. 267-278, 2005.

\bibitem{k22} {\sc  V. Mazya, T. Shaposhnikova and J. Hadamard}, {\em A Universal Mathematician}, American Mathematical Society, 1998.

 \bibitem{r3} {\sc A. Mejdi, F. Belgacem and H. Fekih}, {\em On Cauchy's problem: II. Completion, regularization and approximation}, Inverse Problems, vol. 22, pp. 1307-1336, 2006.
 
\bibitem{r5} {\sc   A. Tikhonov and V. Arsenine}, {\em  Mthode de Resolution de Problmes mal poses}, Editions Mir, 1976.

 
\bibitem{s10} {\sc A. Tikhonov and V. Arsenin}, {\em Solutions of Ill-Posed Problems}, Winston and Sons, Washington, 1977.

 
\bibitem{w2018} {\sc C. Wang}, {\em A New Primal-Dual Weak Galerkin
  Finite Element Method for Ill-posed Elliptic Cauchy Problems},
 Journal of Computational and Applied Mathematics, vol 371, 112629, 2020.

\bibitem{ww2016} {\sc C. Wang and J. Wang}, {\em A primal-dual weak Galerkin finite element method for second order elliptic equations in non-divergence form},  Mathematics of Computation, Math. Comp., vol. 87, pp. 515-545, 2018.

\bibitem{ww2018} {\sc C. Wang and J. Wang}, {\em Primal-Dual Weak Galerkin Finite Element Methods for Elliptic Cauchy Problems}, Computers and Mathematics with Applications, vol 79(3), pp. 746-763, 2020.

 
\bibitem{ww2017} {\sc C. Wang and J. Wang}, {\em A Primal-Dual weak Galerkin finite element method for Fokker-Planck type equations}, arXiv:1704.05606, SIAM Journal of Numerical Analysis, accepted.

\bibitem{tr1} {\sc C. Wang and J. Wang}, {\em A Primal-Dual Finite Element Method for First-Order Transport Problems}, Journal of Computational Physics, Vol. 417, 109571, 2020.
  
\bibitem{wy3655} {\sc J. Wang and X. Ye}, {\em A weak Galerkin mixed finite element method for second-order elliptic problems},  Math. Comp., vol. 83, pp. 2101-2126, 2014.

 
\bibitem{s12} {\sc T. Wei, Y. Hon and L. Ling}, {\em  Method of fundamental solutions with regularization techniques for Cauchy problems of elliptic operators}, Eng. Anal. Bound. Elem., vol. 31, pp. 373-385, 2007.

 \bibitem{s17} {\sc T. Wei, H. Qin and R. Shi}, {\em Numerical solution of an inverse 2D Cauchy problem connected with the Helmholtz equation}, Inverse Problems, vol. 24, pp. 1-18, 2008.
 
  
\bibitem{ludmil} {\sc C. Wang and L. Zikatanov}, {\em Low Regularity Primal-Dual Weak Galerkin Finite Element Methods for Convection-Diffusion Equations}. arXiv:1901.06743.
 
 


\end{thebibliography}
\end{document}